\documentclass[11pt]{article}

\usepackage{amsfonts}
\usepackage{amssymb}
\usepackage{amsthm}
\usepackage{amsmath}
\usepackage{bm}
\usepackage{amscd}

\usepackage{hyperref}

\numberwithin{equation}{section} \DeclareMathSizes{2}{10}{12}{13}
\usepackage[all]{xy}

\newtheorem{thm}{Proposition}[section]

\newtheorem{lem}[thm]{Lemma}
\newtheorem{defn}[thm]{Definition}

\title{
Pure resolutions of unbounded complexes of modules}

\smallskip 

\smallskip 
\author{Abhishek Banerjee}
\date{ }

\begin{document}

\maketitle

\centerline{\emph{Dept. of Mathematics, Indian Institute of Science, Bangalore-560012, India.}}
\centerline{\emph{Email: abhishekbanerjee1313@gmail.com}}

\medskip

\begin{abstract}
Let $R$ be a commutative ring. We show that pure injective resolutions and pure projective
resolutions can be constructed for unbounded complexes of $R$-modules. We use
these to obtain a closed symmetric monoidal structure on the unbounded pure derived category. 
\end{abstract}

\medskip

\medskip
\emph{MSC (2010) Subject Classification:}  13D09, 16E35, 18D10.

\medskip
\emph{Keywords:}  Pure resolutions, pure projectives, pure injectives, pure acyclic. 

\medskip

\medskip

\section{Introduction}

\medskip

\medskip
Given a ring $R$, there are two well known exact structures (in the sense of Quillen \cite{Quillen}) on the category
of $R$-modules: the usual exact structure and the pure exact structure (see \cite{Jal}). The usual derived
category $D(R)$, which is constructed by inverting quasi-isomorphisms in the homotopy category of $R$-modules,
has been studied extensively in homological algebra. Additionally, there has been recent interest by several
authors (see, for instance, \cite{JDC}, \cite{EsOd}, \cite{Gill}, \cite{Nee2}, \cite{Serpe}, \cite{Jal}) in studying the pure exact structure on the category of $R$-modules. This raises the question of which of the  properties of
the usual derived category $D(R)$ can also be extended to the ``pure derived category'' $D_{pur}(R)$, which is
obtained by inverting pure quasi-isomorphisms in the homotopy category of $R$-modules. When $R$
is a commutative ring, the purpose
of this paper is to construct pure injective and pure projective resolutions of complexes of
$R$-modules. We then use these resolutions to exhibit a closed symmetric monoidal structure 
on the category $D_{pur}(R)$. 

\medskip
It is important to note that in this article, we work with unbounded complexes of $R$-modules and consider
the unbounded pure derived category $D_{pur}(R)$. The question behind this paper was motivated naturally
by reading the recent work of Zheng and Huang \cite{Jal}, where the authors study pure  resolutions
of bounded complexes. Further, the authors in \cite{Jal} have also shown that any bounded above  
(resp. bounded below) complex of $R$-modules admits a pure injective (resp. pure projective) resolution. However,
the general question of pure resolutions for arbitrary unbounded complexes is still left open 
in \cite{Jal}. Since our methods require us to have a closed symmetric monoidal structure on
the category of $R$-modules, we will limit ourselves to commutative rings. It should be mentioned here
that even in the case of the classical derived category $D(R)$, the construction of projective and injective
resolutions of arbitrary unbounded complexes presents some difficulties. In \cite{Sp}, Spaltenstein showed
the existence of projective and injective resolutions for unbounded complexes, which allowed him
to  remove the boundedness conditions for the existence of certain derived functors of functors such as Hom and 
tensor product. The derived Hom and derived tensor product functors on the unbounded  derived category were also constructed  by B\"{o}kstedt and Neeman in \cite{BN},
 where the authors
used the method of homotopy colimits. For more  on resolutions of unbounded complexes, the reader
may also see the work of Alonso Tarr\'{i}o et al. \cite{Alt} and Serp\'{e} \cite{Serpe}. Our methods in this paper
are a combination of the classical method of constructing resolutions of bounded complexes (see, for instance,
\cite{Murf1}) along with the techniques of Spaltenstein \cite{Sp} for treating arbitrary unbounded complexes. 

\medskip
We now describe the structure of the paper in detail. In Section 2, we briefly recollect the notions of pure acyclic
complexes, pure projective modules, pure injective modules, pure quasi-isomorphisms and pure resolutions that 
we will need in the paper. In Section 3, we show how to construct pure injective resolutions for unbounded
complexes of $R$-modules. Thereafter, pure projective resolutions of unbounded complexes are constructed
in Section 4. It should be noted that due to the fact that tensoring preserves cokernels but not kernels, 
we will need to somewhat adjust our methods in Section 4, i.e., our arguments for pure projective resolutions
are not exactly the dual of our arguments for pure injective resolutions. Finally, in Section 5, we use
pure projective and pure injective resolutions to obtain a ``pure derived Hom functor'':
\begin{equation}\label{IHOM}
PHom^\bullet_R(\_\_,\_\_): D_{pur}(R)^{op}\times D_{pur}(R)\longrightarrow D_{pur}(R)
\end{equation} Then, in the case of a commutative ring $R$, \eqref{IHOM} removes the boundedness condition for the existence
of the pure derived Hom functor in \cite[$\S$ 4]{Jal}. We conclude by showing that we have natural isomorphisms:
\begin{equation}\label{Iqe5.10}
\begin{array}{c}
PHom_R^\bullet((A^\bullet\otimes_RB^\bullet)^\bullet,C^\bullet)\cong PHom_R^\bullet (A^\bullet,P
Hom^\bullet_R(B^\bullet,C^\bullet))\\
Hom_{D_{pur}(R)}((A^\bullet\otimes_RB^\bullet)^\bullet,C^\bullet)\cong Hom_{D_{pur}}(A^\bullet,P
Hom^\bullet_R(B^\bullet,C^\bullet))\\
\end{array}
\end{equation}
for any complexes $A^\bullet$, $B^\bullet$ and $C^\bullet$ of $R$-modules, thus giving the unbounded pure derived category $D_{pur}(R)$ the structure of a closed symmetric monoidal category. 

\medskip

\medskip
{\bf Acknowledgements:} The author is grateful for the hospitality of the Academy of Mathematics and Systems Science at the Chinese
Academy of Sciences in Beijing, where some of this paper was written.

\medskip

\medskip

\section{Pure acyclic complexes}

\medskip

\medskip 
In this section, we will briefly recall some facts on pure acyclic complexes, pure projectives, pure injectives and pure resolutions
that we will use in the rest of the paper. 
Throughout, we let $R$ be a commutative  ring with identity. We denote by $R-Mod$  the category
of   $R$-modules. We will denote by $C(R)$ the category of cochain complexes:
\begin{equation}
 \begin{CD} (M^\bullet,d_M^\bullet) : \qquad
\dots @>>> M^{i-1} @>d^{i-1}_M>> M^i @>d^i_M>> M^{i+1} @>>> \dots 
\end{CD}
\end{equation} of  $R$-modules and by $K(R)$ the homotopy category of $R-Mod$. By abuse of notation,
we will usually denote an object $(M^\bullet,d_M^\bullet)\in C(R)$ simply by $M^\bullet$. For any
$n\in \mathbb Z$, we let $M[n]^\bullet$ denote the shifted cochain complex given by
$M[n]^i:=M^{i+n}$ with differential $d_{M[n]}^i=(-1)^nd_M^{i+n}$. The unbounded derived category of $R$-modules  is denoted
by $D(R)$.  A cochain complex
$M^\bullet$ is said to be bounded above (resp. bounded below) if $M^i$ vanishes for all $i$ sufficiently
large (resp. for all $i$ sufficiently small).  Given complexes $(M^\bullet,d_M^\bullet)$ and $(N^\bullet,
d_N^\bullet)\in C(R)$, we have an internal Hom object $(Hom_R^\bullet(M^\bullet,N^\bullet),d^\bullet)\in C(R)$ given by:
\begin{equation*}
\begin{array}{c}
Hom_R^i(M^\bullet,N^\bullet):=\underset{j\in \mathbb Z}{\prod}Hom_R(M^j,N^{i+j})\quad\forall\textrm{ }i\in \mathbb Z\\
d^i(f):=d_N\circ f-(-1)^if\circ d_M\quad \forall\textrm{ }f\in Hom_R^i(M^\bullet,N^\bullet)\\
\end{array}
\end{equation*} We now recall the following definitions.

\medskip
\begin{defn}\label{pureacy} (see \cite{War}) (a) A monomorphism $f:M'\longrightarrow M$ in $R-Mod$ is said to be  pure  if
the induced morphism $N\otimes_Rf:N\otimes_RM'\longrightarrow N\otimes_RM$ is a monomorphism for 
each module $N\in R-Mod$. 

\medskip
(b) A short exact sequence :
\begin{equation}
0\longrightarrow M'\longrightarrow M\longrightarrow M''\longrightarrow 0 
\end{equation} in $R-Mod$ is said to be pure acyclic if the induced sequence 
\begin{equation}
0\longrightarrow N\otimes_RM'\longrightarrow N\otimes_RM\longrightarrow N\otimes_RM''\longrightarrow 0 
\end{equation} is acyclic for each module $N\in R-Mod$. In such a situation, the morphism
$M\longrightarrow M''$ is said to be a pure epimorphism. 

\medskip
(c) (see \cite{Gr99}) More generally, a complex $M^\bullet\in C(R)$ is said to be pure acyclic if  the induced sequence
$N\otimes_RM^\bullet$ is acyclic for each module $N\in R-Mod$. This is equivalent to the complex $Hom^\bullet_R(F,M^\bullet)$
being acyclic for each finitely presented $R$-module $F$. 

\medskip
(d) An acyclic  complex  $(M^\bullet,d_M^\bullet)\in C(R)$  is said to be pure acyclic at some given $n\in \mathbb Z$ if 
\begin{equation}
0\longrightarrow Im(d_M^{n-1})=Ker(d_M^n)\longrightarrow M^n\longrightarrow Coker(d_M^{n-1})\longrightarrow 0 
\end{equation} is a pure short exact sequence in the sense of (b). The  complex  $(M^\bullet,d_M^\bullet)\in C(R)$ is pure acyclic in the sense of
(c) if and only if
it is pure acyclic at each $n\in \mathbb Z$.
\end{defn} 

\medskip
Given a morphism $f^\bullet:M^\bullet\longrightarrow N^\bullet$, its mapping cone $C_f^\bullet$ is taken to be the complex
$C_f^\bullet:=M[1]^\bullet\oplus N^\bullet$ with differential $d_{C_f}^\bullet$ given by
\begin{equation}
d_{C_f}^i=\begin{pmatrix} -d_M^{i+1} & 0 \\ f^{i+1} & d_N^i \\
\end{pmatrix} :M^{i+1}\oplus N^i\longrightarrow M^{i+2}\oplus N^{i+1}
\end{equation} We notice that
the canonical projections $p^i:M^i\oplus N^{i-1}\longrightarrow M^i$  determine a morphism of complexes
from $C_f^\bullet$ to $M[1]^\bullet$.

\medskip
\begin{defn} (see \cite[Definition 2.7]{Jal}) A morphism $f^\bullet:M^\bullet\longrightarrow N^\bullet$ 
in $C(R)$ is   a pure quasi-isomorphism if its cone $C_f^\bullet$ is pure acyclic. 

\medskip
Equivalently, $f^\bullet$ is a pure quasi-isomorphism if  $M'\otimes_Rf^\bullet:M'\otimes_RM^\bullet\longrightarrow M'\otimes_RN^\bullet$ is a 
quasi-isomorphism for each module $M'\in R-Mod$. 

\end{defn} 

\medskip
\begin{defn} (see \cite{War}) A module $P\in R-Mod$ is  pure projective if the functor $Hom_R^\bullet(P,\_\_)$ carries 
pure acyclic complexes to pure acyclic complexes.
 Similarly, a module $I\in R-Mod$ is   pure
injective if the functor $Hom_R^\bullet(\_\_,I)$ preserves pure acyclic complexes.

\medskip
 The category of pure projectives in $R-Mod$ will be denoted
by $\mathcal{PP}$ and the category of pure injectives in $R-Mod$ by
$\mathcal{PI}$. 
\end{defn}

\medskip
We mention here (see \cite[Remark 2.6]{Jal}) that a complex $M^\bullet\in C(R)$ is pure acyclic if and only
if $Hom^\bullet_R(P,M^\bullet)$ is  acyclic for any pure projective $P\in \mathcal{PP}$. This is also equivalent
to $Hom^\bullet_R(M^\bullet,I)$ being  acyclic for any pure injective $I\in \mathcal {PI}$. 

\medskip
On the other hand, (see \cite[Remark 2.8]{Jal})  a morphism $f^\bullet:M^\bullet\longrightarrow N^\bullet$ of  complexes
is a pure quasi-isomorphism if and only if $Hom^\bullet_R(P,f^\bullet):Hom^\bullet_R(P,M^\bullet)
\longrightarrow Hom^\bullet_R(P,N^\bullet)$ (resp. $Hom^\bullet_R(f^\bullet,I): 
Hom^\bullet_R(N^\bullet,I)\longrightarrow Hom^\bullet_R(M^\bullet,I)$ ) is a 
quasi-isomorphism for each $P\in \mathcal{PP}$ (resp. for each $I\in \mathcal{PI}$). 

\medskip
\begin{defn}\label{D2.4} (see \cite{Jal}) (a) Let $M^\bullet\in C(R)$. A morphism $f^\bullet:P^\bullet\longrightarrow M^\bullet$
is said to be a pure projective resolution of $M^\bullet$ if it satisfies the following conditions:

\smallskip
$\hspace{0.1in}$(i)  $P^\bullet$ is a complex of pure projective modules and $f^\bullet$ is a pure quasi-isomorphism. 

\smallskip
$\hspace{0.1in}$(ii) The functor $Hom^\bullet_R(P^\bullet,\_\_)$ preserves pure acyclic complexes. 

\medskip
(b) Dually, a morphism $g^\bullet:M^\bullet\longrightarrow I^\bullet$ is said to be a pure injective
resolution of $M^\bullet$ if it satisfies:

\smallskip
$\hspace{0.1in}$ (i)  $I^\bullet$ is a complex of pure injective modules and $g^\bullet$ is a pure quasi-isomorphism. 

\smallskip
$\hspace{0.1in}$(ii) The functor $Hom_R^\bullet(\_\_,I^\bullet)$ preserves pure acyclic complexes. 

\end{defn}

\medskip
Given an $R$-module $M$, it is known (see \cite[Corollary 6]{War}) that there exists a pure injective module $I\in \mathcal{PI}$
and a pure monomorphism $M\hookrightarrow I$. In fact, this holds more generally in any locally finitely presented
additive category (see Herzog \cite{Herzog}). The aim of this paper  is to show that any unbounded complex 
of $R$-modules admits a ``pure projective resolution'' and a ``pure injective resolution''. In \cite{Jal}, Zheng and Huang have already shown that 
any bounded above (resp. bounded below) complex in $C(R)$ admits a pure injective resolution 
(resp. a pure projective resolution) . The proof of Zheng and Huang
in \cite{Jal} uses the techniques of homotopy (co)limits due to B\"{o}kstedt and Neeman \cite{BN} (see
also \cite{Neebook}). However, our proof
will use a combination of the classical technique for constructing  resolutions along with the methods of
Spaltenstein \cite{Sp} for treating arbitrary unbounded complexes.

\medskip

\section{Pure injective resolutions of unbounded complexes}

\medskip

\medskip
In this section, we will need one more concept, that of  a ``$K$-pure injective complex'', which will be
analogous to the classical notion of  a    $K$-injective complex (see, for
instance,  \cite[$\S$ 1]{Sp}).  Although this notion already appears implicitly in Definition \ref{D2.4},  we propose the following definition, which does not seem
to have appeared explicitly before in the literature. 

\medskip
\begin{defn} We will say that a cochain complex $M^\bullet\in C(R)$ is $K$- pure injective if 
the functor $Hom_R^\bullet(\_\_,M^\bullet):C(R)\longrightarrow C(R)$ takes pure
acyclic complexes to pure acyclic complexes. 
\end{defn} 

\medskip
\begin{thm}\label{P3.15} Let $M^\bullet\in C(R)$. Then, the following statements are equivalent:

\medskip
(a) $M^\bullet$ is a $K$-pure injective complex. 

\medskip
(b) For any pure acyclic complex $A^\bullet\in C(R)$, we have $Hom_{K(R)}(A^\bullet,M^\bullet)=0$. 
\end{thm}

\begin{proof} (a) $\Rightarrow$ (b): For any two complexes $B^\bullet$, $C^\bullet\in C(R)$, it is well known
(see, for instance, \cite[$\S$0.4]{Sp}) that 
\begin{equation}\label{hom} 
H^k(Hom_R^\bullet(B^\bullet,C^\bullet))=Hom_{K(R)}(B^\bullet,C[k]^\bullet)\qquad\forall\textrm{ }k\in \mathbb Z
\end{equation} Let $A^\bullet\in Ch(R)$ be pure acyclic. Since $M^\bullet$ is $K$-pure injective, it follows that $Hom_R^\bullet(A^\bullet,
M^\bullet)$ is pure acyclic and, in particular, acyclic. Then, $Hom_{K(R)}(A^\bullet,M^\bullet)=H^0(Hom_R^\bullet(A^\bullet,M^\bullet))=0$. 

\medskip
(b) $\Rightarrow$ (a): Shifting the indices in \eqref{hom}, we obtain:
\begin{equation}\label{homp}
H^k(Hom_R^\bullet(B^\bullet,C^\bullet))=Hom_{K(R)}(B[-k]^\bullet,C^\bullet)\qquad\forall\textrm{ }k\in \mathbb Z
\end{equation} for any two complexes $B^\bullet$, $C^\bullet\in C(R)$. Now let $M^\bullet\in C(R)$ be such that $Hom_{K(R)}(A^\bullet,M^\bullet)=0$
for every pure acyclic $A^\bullet\in C(R)$.  We have to show that $Hom_R^\bullet(A^\bullet,M^\bullet)$ is pure acyclic,
or equivalently that $Hom_R^\bullet(F,Hom_R^\bullet(A^\bullet,M^\bullet))\cong Hom_R^\bullet
(F\otimes_RA^\bullet,M^\bullet)$ is acyclic for every finitely presented $R$-module $F$. If $A^\bullet$ is pure acyclic,
it is immediate from Definition   \ref{pureacy} that $(F\otimes_R A^\bullet)[-k]^\bullet$ is also pure acyclic for any
$k\in \mathbb Z$. It now follows from \eqref{homp} that:
\begin{equation}
H^k(Hom_R^\bullet
(F\otimes_RA^\bullet,M^\bullet))=Hom_{K(R)}((F\otimes_R A^\bullet)[-k]^\bullet,M^\bullet)=0\qquad\forall\textrm{ }k\in \mathbb Z
\end{equation} This proves the result. 

\end{proof} 

\medskip
Let $K_{pac}(R)$ denote the full subcategory of  $K(R)$ consisting of complexes that are also pure acyclic. The cone of a morphism of pure acyclic complexes is still pure acyclic and hence $K_{pac}(R)$ is a triangulated subcategory.  Then, since 
 $K_{pac}(R)$ is closed under direct summands, it follows from Rickard's criterion \cite[Proposition 1.3]{Rick} (see also \cite{Nee1}) that $K_{pac}(R)$ is a thick
subcategory. We then consider  the ``pure derived category'' $D_{pur}(R)$ as in \cite[$\S$3]{Jal} given by the Verdier quotient:
\begin{equation}
D_{pur}(R):=K(R)/K_{pac}(R)
\end{equation} Consider a ``right roof'' in $K(R)$ given by a pair of morphisms $(f^\bullet,u^\bullet)$ of the form: 
\begin{equation}\label{rroof}
A^\bullet \overset{f^\bullet}{\longrightarrow} C^\bullet \overset{u^\bullet}{\Longleftarrow} B^\bullet
\end{equation} where $u^\bullet$ is a pure quasi-isomorphism. Then, the morphisms in $D_{pur}(R)$ from 
$A^\bullet$ to $B^\bullet$ can be given by equivalence classes of right roofs as in \eqref{rroof} (see \cite{GaZi}
for details). We will now characterize $K$-pure injective complexes in terms of the pure derived category
$D_{pur}(R)$. 

\medskip
\begin{thm}\label{P3.16} Let $M^\bullet\in C(R)$. Then, the following statements are equivalent:

\medskip
(a) $M^\bullet$ is a $K$-pure injective complex. 

\medskip
(b) For any  complex $A^\bullet\in C(R)$, we have  $Hom_{K(R)}(A^\bullet,M^\bullet)\cong Hom_{D_{pur}(R)}(
A^\bullet,M^\bullet)$. 
\end{thm}

\begin{proof}  (b) $\Rightarrow$ (a): Let $A^\bullet\in C(R)$ be pure acyclic. Then, $A^\bullet =0$
in the pure derived category $D_{pur}(R)$. Hence, $Hom_{K(R)}(A^\bullet,M^\bullet)\cong Hom_{D_{pur}(R)}(A^\bullet,M^\bullet)=0$ 
and it follows from Proposition \ref{P3.15} that $M^\bullet $ is $K$-pure injective.

\medskip
(a) $\Rightarrow$ (b): Suppose that $M^\bullet$ is a $K$-pure injective complex. We first show that if
$u^\bullet:M^\bullet\longrightarrow N^\bullet$ is a pure quasi-isomorphism, it must have a left inverse up to 
homotopy, i.e., we must have some $v^\bullet:N^\bullet
\longrightarrow M^\bullet$ such that $v^\bullet\circ u^\bullet \sim 1$.  If $C_u^\bullet$ denotes the cone of
$u^\bullet$, applying the functor $Hom_{K(R)}(\_\_,M^\bullet)$ to the distinguished triangle
$M^\bullet\overset{u^\bullet}{\longrightarrow}N^\bullet \longrightarrow C_u^\bullet$ gives an induced exact sequence:
\begin{equation*}
Hom_{K(R)}(C_u^\bullet,M^\bullet)\rightarrow Hom_{K(R)}(N^\bullet,M^\bullet)
\rightarrow Hom_{K(R)}(M^\bullet,M^\bullet)
\rightarrow Hom_{K(R)}(C_u[-1]^\bullet,M^\bullet) 
\end{equation*} Since $u^\bullet$ is a pure quasi-isomorphism, its cone  $C_u^\bullet$ is pure acyclic and $M^\bullet$
being $K$-pure injective, we get $Hom_{K(R)}(C_u^\bullet,M^\bullet)=Hom_{K(R)}(C_u[-1]^\bullet,M^\bullet)=0$. 
This gives us an isomorphism $Hom_{K(R)}(u^\bullet,M^\bullet): Hom_{K(R)}(N^\bullet,M^\bullet)
\overset{\cong}{\longrightarrow} Hom_{K(R)}(M^\bullet,M^\bullet)$. Choosing $1\in Hom_{K(R)}(M^\bullet,M^\bullet)$, 
it follows that there exists  $v^\bullet:N^\bullet
\longrightarrow M^\bullet$ such that $v^\bullet\circ u^\bullet \sim 1$. Now, given a morphism 
 $(f^\bullet,u^\bullet)\in Hom_{D_{pur}(R)}(A^\bullet,M^\bullet)$ having the form of a right roof 
\begin{equation}\label{rroof1}
A^\bullet \overset{f^\bullet}{\longrightarrow} N^\bullet \overset{u^\bullet}{\Longleftarrow} M^\bullet
\end{equation} we associate  $(f^\bullet,u^\bullet)\in Hom_{D_{pur}(R)}(A^\bullet,M^\bullet)$ to
$v^\bullet\circ f^\bullet \in Hom_{K(R)}(A^\bullet, M^\bullet)$. Conversely, any morphism $g^\bullet
\in Hom_{K(R)}(A^\bullet, M^\bullet)$ is associated to the roof $(g^\bullet,1)\in  Hom_{D_{pur}(R)}(A^\bullet,M^\bullet)$. 
Since the right roofs $(v^\bullet\circ f^\bullet,1)$, $ (f^\bullet,u^\bullet)$ are equivalent
in $Hom_{D_{pur}(R)}(A^\bullet,M^\bullet)$, it is clear that these two associations are inverse to each other. Hence, we have $Hom_{K(R)}(A^\bullet,M^\bullet)\cong Hom_{D_{pur}(R)}(
A^\bullet,M^\bullet)$. 

\end{proof} 

\medskip
\begin{thm}\label{P3.2}  Let $(I^\bullet,d_I^\bullet)\in C(R)$ be a bounded below complex of pure injective modules. Then, $I^\bullet$
is K-pure injective.
\end{thm}

\begin{proof} For the sake of definiteness, we suppose that $I^j=0$ for every $j<0$ and $I^0\ne 0$.   Using Proposition \ref{P3.15}, it suffices to show that
$Hom_{K(R)}(A^\bullet,I^\bullet)=0$ for any pure acyclic $(A^\bullet,d_A^\bullet)\in C(R)$. We will show that any morphism
$f:(A^\bullet,d_A^\bullet)\longrightarrow (I^\bullet,d_I^\bullet)$ in $C(R)$ is null-homotopic. For some given
integer $K$, we suppose that we have constructed
maps $t^{k}:A^{k}\longrightarrow I^{k-1}$ for all integers $k\leq K$ such that $t^k\circ d_A^{k-1}+d_I^{k-2}\circ t^{k-1}
=f^{k-1}$. This is already true for $K=0$. We now note that:
\begin{equation*}
(f^k-d_I^{k-1}\circ t^k)\circ d_A^{k-1}=f^k\circ d_A^{k-1}-d_I^{k-1}\circ (t^k\circ d_A^{k-1})
=(f^k\circ d_A^{k-1}-d_I^{k-1}\circ f^{k-1})+d_I^{k-1}\circ d_I^{k-2}\circ t^{k-1}=0
\end{equation*} Hence, the morphism $(f^k-d_I^{k-1}\circ t^k): A^k\longrightarrow I^k$ factors through
$A^k/Im(d_A^{k-1})=A^k/Ker(d_A^k)\cong Im(d_A^k)$. From Definition \ref{pureacy}, it follows that
$Im(d_A^k)\hookrightarrow A^{k+1}$ is a pure monomorphism.  Since $I^k$ is pure injective, the morphism 
$Im(d_A^k)\longrightarrow I^k$ extends to a morphism $t^{k+1}:A^{k+1}\longrightarrow I^k$ that satisfies
$(f^k-d_I^{k-1}\circ t^k)=t^{k+1}\circ d_A^k$. This proves the result. 

\end{proof}

\medskip
As mentioned in Section 2, it is well known (see \cite[Corollary 6]{War}) that any $R$-module $M$ admits a pure 
monomorphism $M\hookrightarrow I$ into a pure injective module $I$. 
We  will now show that any bounded below complex of $R$-modules admits a pure injective resolution. 

\medskip 
\begin{thm}\label{P3.3} Let $M^\bullet\in C(R)$ be a cochain complex of $R$-modules that is bounded below. Then, there exists
a pure quasi-isomorphism $u^\bullet:M^\bullet\longrightarrow I^\bullet$ such that $I^\bullet$ is a bounded
below complex of pure injective modules.
\end{thm}

\begin{proof} For the sake of definiteness, we suppose that the complex 
$(M^\bullet,d^\bullet)$ satisfies $M^j=0$ for each $j<0$  but $M^0\ne 0$. We put $I^j=0$
for each $j<0$. We choose a pure monomorphism $u^0:M^0\hookrightarrow I^0$ with $I^0$ pure injective. Then, for every $i<1$, we have already
constructed pure injective modules $I^i$, morphisms $u^i:M^i\longrightarrow I^i$ and differentials $e^{i-1}:I^{i-1}\longrightarrow I^i$ such that we have induced isomorphisms:
\begin{equation}
H^{i-1}(M^\bullet\otimes_R N) \overset{\sim}{\longrightarrow} Ker(e^{i-1}\otimes_RN)/Im(e^{i-2}\otimes_RN)
\qquad\forall\textrm{ }N\in R-Mod
\end{equation} 
and monomorphisms:
\begin{equation}
Coker(d^{i-1}\otimes_RN)
\hookrightarrow Coker(e^{i-1}\otimes_RN)\qquad\forall\textrm{ }N\in R-Mod
\end{equation} We suppose that we have already done this for all $i\in \mathbb Z$ less than some given integer $k\geq 1$. We will now show
that we can choose a pure injective $I^k$, a morphism $u^k:M^k\longrightarrow I^k$  and a differential $e^{k-1}:I^{k-1}\longrightarrow I^k$ such that: 
\begin{equation}
\begin{array}{c}
H^{k-1}(M^\bullet\otimes_RN) \overset{\sim}{\longrightarrow} Ker(e^{k-1}\otimes_RN)/Im(e^{k-2}\otimes_RN)
\qquad\forall\textrm{ }N\in R-Mod\\
Coker(d^{k-1}\otimes_RN)
\hookrightarrow Coker(e^{k-1}\otimes_RN)\qquad\forall\textrm{ }N\in R-Mod\\
\end{array}
\end{equation} For this we consider the colimit $C^k$ defined by the following pushout square:
\begin{equation}
\begin{CD}
M^{k-1}@>d^{k-1}>> M^k \\
@VVV @VVV \\
Coker(e^{k-2}) @>>> C^k\\
\end{CD}
\end{equation} and choose a pure monomorphism $C^k\hookrightarrow I^k$ with $I^k$ pure injective. Then, for any 
$N\in R-Mod$, $C^k\otimes_RN\longrightarrow I^k\otimes_RN$ is a monomorphism and the following square is
still a pushout:
\begin{equation}\label{3.6y}
\begin{CD}
M^{k-1}\otimes_RN@>d^{k-1}\otimes_RN>> M^k\otimes_RN \\
@VVV @VVV \\
Coker(e^{k-2})\otimes_RN=Coker(e^{k-2}\otimes_RN) @>>> C^k\otimes_RN\\
\end{CD}
\end{equation} We now define $e^{k-1}:I^{k-1}\longrightarrow I^k$ to be the composition
$I^{k-1}\longrightarrow Coker(e^{k-2})\longrightarrow C^k\hookrightarrow I^k$ and $u^k:M^k\longrightarrow 
I^k$ to be the composition
$M^k\longrightarrow C^k\longrightarrow I^k$. Applying the dual of 
\cite[Lemma 68]{Murf1} to the pushout square \eqref{3.6y} along with the monomorphism 
$ C^k\otimes_RN\hookrightarrow I^k\otimes_RN$ gives us a monomorphism $Coker(d^{k-1}\otimes_RN)
\hookrightarrow Coker(e^{k-1}\otimes_RN)$.

\medskip 
We now put $d^j_N:=d^j\otimes_RN$ and $e^j_N:=e^j\otimes_RN$ for any integer $j$ and any $R$-module $N$. 
Since the morphisms 
$M^{k-1}\otimes_RN\longrightarrow M^k\otimes_RN$ and $M^{k-1}\otimes_RN\longrightarrow Coker(e^{k-2}\otimes_RN)$
both factor through the epimorphism $M^{k-1}\otimes_RN\longrightarrow (M^{k-1}\otimes_RN)/Im(d^{k-2}_N)$, we can
simply replace $M^{k-1}\otimes_RN$ by $(M^{k-1}\otimes_RN)/Im(d^{k-2}_N)$ in \eqref{3.6y} and still obtain a pushout square. Now if
we let $C_N$ be the colimit of the system $Coker(e^{k-2}_N)\longleftarrow (M^{k-1}\otimes_RN)/Im(d^{k-2}_N)\longrightarrow Im(d^{k-1}_N)$, we obtain the following commutative diagram:
\begin{equation}
\begin{CD}
(M^{k-1}\otimes_RN)/Im(d^{k-2}_N)=Coker(d^{k-2}_N)@>>> Im(d^{k-1}_N) @>>> M^k\otimes_RN \\
@VVV @VVV @VVV \\
Coker(e^{k-2}_N) @>>> C_N  @>>> C^k\otimes_RN \\
\end{CD}
\end{equation} where all the squares are pushouts. The pushout of the epimorphism
$(M^{k-1}\otimes_RN)/Im(d^{k-2}_N)\longrightarrow Im(d^{k-1}_N)$ gives an epimorphism
$Coker(e^{k-2}_N)\twoheadrightarrow C_N $. On the other hand, since $R-Mod$ is an abelian
category, it follows that the pushout of the monomorphism $Im(d^{k-1}_N)\hookrightarrow M^k\otimes_RN$
gives a monomorphism $C_N\hookrightarrow C^k\otimes_RN$. Accordingly, the morphism $e^{k-1}_N:I^{k-1}\otimes_RN
\longrightarrow I^k\otimes_RN$ can be factored as the epimorphism:
\begin{equation}\label{3.8y}
I^{k-1}\otimes_RN\twoheadrightarrow (I^{k-1}\otimes_RN)/Im(e^{k-2}_N)=Coker(e^{k-2}_N)\twoheadrightarrow C_N
\end{equation} followed by the monomorphism:
\begin{equation}\label{3.9y}
C_N\hookrightarrow C^k\otimes_RN\hookrightarrow I^k\otimes_RN
\end{equation} From \eqref{3.8y} and \eqref{3.9y} it follows that $C_N\cong (I^{k-1}\otimes_RN)/Ker(e^{k-1}_N)
=Im(e^{k-1}_N)$. By assumption, we have a monomorphism 
$Coker(d^{k-2}_N)
\hookrightarrow Coker(e^{k-2}_N)$.  This gives us the following pushout square:
\begin{equation}\label{3.10y}
\begin{CD}
(M^{k-1}\otimes_RN)/Im(d^{k-2}_N)=Coker(d^{k-2}_N)@>epic>> Im(d^{k-1}_N)
=(M^{k-1}\otimes_RN)/Ker(d^{k-1}_N) \\
@VmonicVV @VVV  \\
(I^{k-1}\otimes_RN)/Im(e^{k-2}_N)=Coker(e^{k-2}_N) @>epic>> Im(e^{k-1}_N) =(I^{k-1}\otimes_RN)/Ker(e^{k-1}_N)  \\
\end{CD}
\end{equation} The fact that \eqref{3.10y} is a pushout square and that $Coker(d^{k-2}_N)
\longrightarrow Coker(e^{k-2}_N)$ is a monomorphism shows that we have an isomorphism 
of the kernels of the two horizontal morphisms. It is also clear that the kernels of these horizontal
morphisms are $H^{k-1}(M^\bullet\otimes_RN)$ and $Ker(e^{k-1}\otimes_RN)/Im(e^{k-2}\otimes_RN)$ respectively.
Thus, we can construct inductively a pure quasi-isomorphism $M^\bullet\longrightarrow I^\bullet$ where
$I^\bullet$ is a bounded below complex of pure injectives.  

\end{proof}

\medskip
\begin{thm}\label{P3.5} Let $f^\bullet:M_2^\bullet\longrightarrow M_1^\bullet $ be a morphism of bounded below complexes
in $C(R)$. Let $u^\bullet_1: M_1^\bullet\longrightarrow I_1^\bullet$ be a pure quasi-isomorphism from
$M_1^\bullet$ to a bounded below complex $I_1^\bullet$ of pure injectives. Then, there exists a bounded below complex $I_2^\bullet$ of pure injectives,  a pure quasi-isomorphism
$u_2^\bullet: M_2^\bullet\longrightarrow I_2^\bullet$  and a morphism $g^\bullet:I_2^\bullet\longrightarrow I_1^\bullet$  
fitting into a commutative diagram: 
\begin{equation}\label{tr3.12}
\begin{CD}
M_2^\bullet @>u_2^\bullet>> I_2^\bullet \\
@Vf^\bullet VV @Vg^\bullet VV\\
M_1^\bullet @>u_1^\bullet>>  I_1^\bullet\\
\end{CD}
\end{equation}
\end{thm}

\begin{proof} It is clear that the mapping cone $C_{u_1f}^\bullet$ of the composition 
$u_1^\bullet\circ f^\bullet : M_2^\bullet\longrightarrow I_1^\bullet$ is a bounded below complex in
$C(R)$. Applying Proposition \ref{P3.3}, we choose a pure quasi-isomorphism $v^\bullet:C^\bullet_{u_1f}
\longrightarrow I^\bullet$ to a bounded below complex of pure injectives. Thereafter, we consider
the composition $h^\bullet:I_1^\bullet\longrightarrow C_{u_1f}^\bullet\overset{v^\bullet}{\longrightarrow}I^\bullet$
and its mapping cone $C^\bullet_h$. We now have the following commutative diagram in $K(R)$:
\begin{equation}\label{tr3.13}
    \xymatrix{
      M_2^\bullet \ar[r]^{u_1^\bullet\circ f^\bullet }\ar@{.>}[d]^{u_2^\bullet} &  I_1^\bullet \ar[r] \ar[d]_1 &  C_{u_1f}^\bullet \ar[d]^{v^\bullet} \ar[r] &  M_2[1]^\bullet \ar@{.>}[d]^{u_2[1]^\bullet} \\
      C_h[-1]^\bullet  \ar[r]^{g^\bullet} & I_1^\bullet \ar[r]_{h^\bullet}       & I^\bullet \ar[r] & C_h^\bullet \\ }
\end{equation}
Since the horizontal rows in \eqref{tr3.13} are distinguished triangles, the triangulated structure
on $K(R)$ implies that we have a   morphism $u_2^\bullet:M_2^\bullet\longrightarrow C_h[-1]^\bullet$ making
the diagram commute. It is clear that $C_h[-1]^\bullet$ is a bounded below complex of pure injectives and 
we set $I_2^\bullet:=C_h[-1]^\bullet$. Now, for any module $N\in R-Mod$, we have an induced commutative
diagram:
\begin{equation}\label{tr3.14}
    \xymatrix{
     N\otimes_R M_2^\bullet \ar[r]^{N\otimes_R(u_1^\bullet\circ f^\bullet) } \ar[d]^{N\otimes_Ru_2^\bullet} &  N\otimes_RI_1^\bullet \ar[r] \ar[d]_1 &  N\otimes_RC_{u_1f}^\bullet \ar[d]^{N\otimes_Rv^\bullet} \ar[r] &  (N\otimes_RM_2^\bullet)[1] \ar[d]^{ (N\otimes_Ru_2^\bullet)[1]} \\
     N\otimes_RI_2^\bullet=N\otimes_R C_h[-1]^\bullet   \ar[r]^>{N\otimes_Rg^\bullet\qquad } & N\otimes_RI_1^\bullet \ar[r]_{N\otimes_Rh^\bullet}       & N\otimes_R I^\bullet \ar[r] & 
(N\otimes_RI_2^\bullet)[1] \\ }
\end{equation} Since $v^\bullet$ is a pure quasi-isomorphism, $N\otimes_R v^\bullet$ is a quasi-isomorphism. 
Since the mapping cone commutes with the functor $N\otimes_R\_\_$,  the horizontal rows in \eqref{tr3.14} are still distinguished
triangles. Now, since
$1:N\otimes_RI_1^\bullet\longrightarrow N\otimes_RI_1^\bullet$ and $N\otimes_R v^\bullet:
 N\otimes_RC_{u_1f}^\bullet \longrightarrow N\otimes_RI^\bullet$ are quasi-isomorphisms (and hence isomorphisms
in the derived category $D(R)$), the third morphism $N\otimes_Ru_2^\bullet:N\otimes_RM_2^\bullet\longrightarrow 
N\otimes_RI_2^\bullet$ is also a quasi-isomorphism (see \cite[Tag 014A]{Stacks}). Hence, $u_2^\bullet:M_2^\bullet
\longrightarrow I_2^\bullet$ is a pure quasi-isomorphism that fits into the commutative square  \eqref{tr3.12}. 

\end{proof}

\medskip
We will now proceed to construct pure injective resolutions for arbitrary, unbounded complexes of $R$-modules. We will
first need the notion of a ``special inverse system''. 

\medskip
\begin{defn}\label{spinv} (see \cite[Definition 2.1]{Sp})  Let $\mathcal T$ be a class of complexes in $C(R)$ that
is closed under isomorphisms.

\medskip

 (a) A $\mathcal T$-special inverse system
of complexes  is an inverse system  $\{I_n^\bullet\}_{n\in \mathbb Z}$ of complexes in $C(R)$ satisfying the 
following conditions for each $n\in \mathbb Z$:

\medskip
(1)  The cochain map $I_n^\bullet\longrightarrow I^\bullet_{n-1}$ is surjective. 

\smallskip
(2) The kernel $K_n^\bullet:=Ker(I_n^\bullet\longrightarrow I_{n-1}^\bullet)$
lies in the class $\mathcal T$. 

\smallskip
(3) The short exact sequence of complexes:
\begin{equation}
0\longrightarrow K^\bullet_n \overset{i_n^\bullet}{\longrightarrow} I_n^\bullet\overset{p^\bullet_n}{\longrightarrow} I_{n-1}^\bullet \longrightarrow 0
\end{equation} is ``semi-split'', i.e., it is split in each degree. 

\medskip
(b) The class $\mathcal T\subseteq C(R)$ is said to be closed under special inverse limits if the inverse limit 
of every
$\mathcal T$-special inverse system in $C(R)$  is contained in $\mathcal T$. 

\end{defn}

\medskip
By slight abuse of notation, we will refer to  $\{I_n^\bullet\}_{n\geq 0}$  as a $\mathcal T$-special inverse system  if setting 
$I_n^\bullet=0$  for all $n<0$ makes 
  $\{I_n^\bullet\}_{n\in \mathbb Z}$ into a $\mathcal T$-special inverse system  in the sense of Definition \ref{spinv} above.

\medskip
\begin{thm}\label{P3.8t} (a) Let $\mathcal C\subseteq C(R)$ be a class of complexes. Let $\mathcal T(\mathcal C)$ denote the class of 
complexes  $M^\bullet\in C(R)$  such that $Hom_R^\bullet(A^\bullet,M^\bullet)$ is pure acyclic for
each $A^\bullet \in \mathcal C$. Then, $\mathcal T(\mathcal C)$ is closed under special inverse limits. 

\medskip
(b) The class of all $K$-pure injective complexes is closed under special inverse limits. 
\end{thm} 

\begin{proof} (a) We begin by setting:
\begin{equation*}
\mathcal T':=\{\mbox{$B^\bullet\in C(R)$ $\vert$ $B^\bullet=F\otimes_RA^\bullet$ for some finitely presented
$R$-module $F$ and some $A^\bullet\in \mathcal C$}\}
\end{equation*} Now since $Hom_R^\bullet(F\otimes_RA^\bullet,M^\bullet)
\cong Hom_R^\bullet (F,Hom_R^\bullet(A^\bullet,M^\bullet))$, a complex $M^\bullet\in \mathcal T(\mathcal C)$
if and only if $Hom_R^\bullet(B^\bullet,M^\bullet)$ is acyclic for each $B^\bullet\in \mathcal T'$. It now follows
from \cite[Corollary 2.5]{Sp} that $\mathcal T(\mathcal C)$ is closed under special inverse limits. The result of 
(b) follows directly from (a) by taking $\mathcal C$ to be the class of all pure acyclic complexes
in $C(R)$. 

\end{proof}

\medskip
Given a complex $(M^\bullet,d^\bullet)\in C(R)$, we recall that for any $n\in \mathbb Z$, its truncation $\tau^{\geq n}M^\bullet$ is given by setting 
\begin{equation}\label{tunc}
(\tau^{\geq n}M)^i=\left\{
\begin{array}{ll}
M^i & \mbox{if $i>n$}\\
Coker(d^{n-1}) & \mbox{if $i=n$} \\
0 & \mbox{if $i<n$} 
\end{array}\right.
\end{equation} Then, it is clear that $H^i(\tau^{\geq n}M^\bullet)=H^i(M^\bullet)$ for all $i\geq n$ and 
$H^i(\tau^{\geq n}M^\bullet)=0$ otherwise. Further,   the canonical morphisms 
$\tau^{\geq n-1}M^\bullet\longrightarrow \tau^{\geq n}M^\bullet$ can be used to express
$M^\bullet$ as an inverse limit $M^\bullet =\underset{n\geq 0}{\varprojlim}\textrm{ }\tau^{\geq -n}M^\bullet$. 

\medskip

\begin{thm}\label{P3.9} For any complex $M^\bullet\in C(R)$ there exists a special inverse system $\{I_n^\bullet\}_{n\geq 0}$
of $K$-pure injective complexes and a morphism $\{f_n:\tau^{\geq -n}M^\bullet\longrightarrow I_n^\bullet\}_{n\geq 0}$ 
of inverse systems 
satisfying the following conditions:

\medskip
(a) Each $I_n^\bullet$ is a bounded below complex of pure injectives. 

\medskip
(b) Each $f_n$ is a pure quasi-isomorphism. 
\end{thm}

\begin{proof}  Using Proposition \ref{P3.3}, we choose a pure quasi-isomorphism $f_0:\tau^{\geq 0}M^\bullet
\longrightarrow I_0^\bullet$ with $I_0^\bullet$ a bounded below complex of pure injectives. For some $n\geq 1$,
we assume that we have already chosen pure quasi-isomorphisms $f_j:\tau^{\geq -j}M^\bullet 
\longrightarrow I_j^\bullet$ for all $0\leq j\leq n-1$ satisfying the required conditions. We now set:
\begin{equation}
\begin{CD}
I^\bullet:=I_{n-1}^\bullet\qquad N^\bullet:=\tau^{\geq -n}M^\bullet \qquad f:N^\bullet=\tau^{\geq -n}M^\bullet
\longrightarrow \tau^{\geq -n+1}M^\bullet @>{f_{n-1}}>> I_{n-1}^\bullet=I^\bullet
\end{CD}
\end{equation} We let $C_f^\bullet$ denote the cone of $f$. Again using Proposition \ref{P3.3}, we choose
a pure quasi-isomorphism $g:C_f^\bullet\longrightarrow J^\bullet$ with $J^\bullet$ a bounded below
complex of pure injective modules. Since $C_f^\bullet=N[1]^\bullet\oplus I^\bullet$ as a $\mathbb Z$-graded module,
$g$ induces morphisms $g':N[1]^\bullet\longrightarrow C_f^\bullet\overset{g}{\longrightarrow} J^\bullet$  and $g'':I^\bullet\longrightarrow C_f^\bullet\overset{g}{\longrightarrow} J^\bullet$ of graded modules and $g''$ is actually a morphism
of complexes. We rewrite $g':N[1]^\bullet\longrightarrow J^\bullet$ as a morphism $g':N^\bullet\longrightarrow J[-1]^\bullet$ and consider
the following morphism for each $i\in \mathbb Z$:
\begin{equation}\label{eq3.22eq}
h^i:N^i\longrightarrow C_{-g''}[-1]^i=I^i\oplus J[-1]^i\qquad n\mapsto (f(n),g'(n))=(f(n),g(n,0)) 
\end{equation} where $C_{-g''}^\bullet$ is the cone of $-g''$. We claim that $h^\bullet=\{h^i\}_{i\in \mathbb Z}$ is a morphism of complexes. For this, we
note that for some $i\in \mathbb Z$ and for any $n\in N^i$ we have:
\begin{equation}
\begin{array}{rl}
h\circ d_N (n) & = (f\circ d_N(n),g(d_N(n),0))\\
d_{C_{-g''}[-1]}\circ h(n)
& =(d_I\circ f(n),g''\circ f(n)-d_J\circ g(n,0))\\
& =(d_I\circ f(n),g''\circ f(n)-g\circ d_{C_f}(n,0)) \\
& = (d_I\circ f(n),g''\circ f(n)-g(-d_N(n),f(n))) \\
& = (f\circ d_N(n),g( d_N(n),0))\\
\end{array}
\end{equation} where $d_N$ is the differential on $N^\bullet$. As a graded module, it is immediate that the cone $C_h^\bullet$ of $h^\bullet$ satisfies:
\begin{equation}\label{3.24fg}  C_h^\bullet=N[1]^\bullet\oplus C_{-g''}[-1]^\bullet
=N[1]^\bullet\oplus I^\bullet\oplus J[-1]^\bullet =(N[2]^\bullet\oplus I[1]^\bullet\oplus J^\bullet )[-1]
=C_{-g}[-1]^\bullet
\end{equation} To show that \eqref{3.24fg} is an isomorphism of complexes, for some given $i\in \mathbb Z$, 
we choose  any 
$x\in N^{i+1}$, $y\in I^i$ and $z\in J^{i-1}$. Then, we see that:
\begin{equation}
\begin{array}{ll}
d_{C_h}(x,y,z) &= (-d_N(x),h(x)+d_{C_{-g''}[-1]}(y,z)) \\
& =(-d_N(x),f(x)+d_I(y),g(x,0)+g''(y)-d_J(z)) \\
&=(-d_N(x),f(x)+d_I(y),g(x,y)-d_J(z))\\
d_{C_{-g}[-1]}(x,y,z)&= -d_{C_{-g}}(x,y,z)\\ & =-(-d_{C_f}(x,y),-g(x,y)+d_J(z))\\
& =(d_{C_f}(x,y),g(x,y)-d_J(z))\\
& =(-d_N(x),f(x)+d_I(y),g(x,y)-d_J(z))\\
\end{array}
\end{equation} Accordingly, we have an isomorphism $C_h^\bullet\cong C_{-g}[-1]^\bullet$ of complexes.
Since $g$ is a pure quasi-isomorphism, it now follows that $C_h^\bullet$ is pure acyclic. In other words,
$h^\bullet:N^\bullet=\tau^{\geq -n}M^\bullet \longrightarrow  C_{-g''}[-1]^\bullet$ is a pure quasi-isomorphism. 
From the definitions, it is clear that $I_n^\bullet:=C_{-g''}[-1]^\bullet$ is a bounded below complex of pure injectives. Finally,
the definition in \eqref{eq3.22eq} makes it clear that $\{I_n^\bullet\}_{n\geq 0}$ is a special inverse system of
$K$-pure injective complexes. This 
proves the result.
\end{proof}

\medskip
\begin{lem}\label{Lem3.10} Let $\{g_n^\bullet:X_n^\bullet\longrightarrow Y_n^\bullet\}_{n\geq 0}$ be a morphism 
of inverse systems $\{X_n^\bullet\}_{n\geq 0} $  and $\{Y_n^\bullet\}_{n\geq 0}$ of complexes of $R$-modules such that
each $g_n^\bullet$ is a quasi-isomorphism. Suppose that for each $i\geq 0$, we can choose a positive integer
$N(i)$ such that the morphisms
\begin{equation}
\tau^{\geq -i}X_{n+1}^\bullet\longrightarrow \tau^{\geq -i}X_n^\bullet
\qquad\qquad \tau^{\geq -i}Y_{n+1}^\bullet\longrightarrow \tau^{\geq -i}Y_n^\bullet
\end{equation} are epimorphisms in each degree for each $n\geq N(i)$. Then, the induced morphism on
the inverse limits $g^\bullet:X^\bullet:=\underset{n\geq 0}{\varprojlim}\textrm{ }X_n^\bullet \longrightarrow Y^\bullet:=\underset{n\geq 0}{\varprojlim}\textrm{ }Y_n^\bullet $ is a quasi-isomorphism. 
\end{lem}

\begin{proof} First we fix some $i\geq 0$. It is clear from the definition of the truncations in \eqref{tunc} that the quasi-isomorphisms
$g_n^\bullet:X_n^\bullet\longrightarrow Y_n^\bullet$ induce quasi-isomorphisms 
$\tau^{\geq -i}g_n^\bullet:\tau^{\geq -i}X_n^\bullet\longrightarrow \tau^{\geq -i}Y_n^\bullet$ of complexes. Further,
the transition maps of the inverse systems $\{\tau^{\geq -i}X_n^\bullet\}_{n\geq N(i)}$ and $\{
\tau^{\geq -i}Y_n^\bullet\}_{n\geq N(i)}$ are
all surjections. It now follows from \cite[Corollary 3.11]{LS} that the induced morphism
\begin{equation}
\underset{n\geq 0}{\varprojlim}\textrm{ }\tau^{\geq -i}X_n^\bullet=\underset{n\geq N(i)}{\varprojlim}\textrm{ }
\tau^{\geq -i}X_n^\bullet\longrightarrow 
\underset{n\geq N(i)}{\varprojlim}\textrm{ }\tau^{\geq -i}Y_n^\bullet=\underset{n\geq 0}{\varprojlim}\textrm{ }\tau^{\geq -i}Y_n^\bullet
\end{equation} is a quasi-isomorphism. From the definition of truncations in \eqref{tunc} it is also clear that:
\begin{equation}
H^j(\underset{n\geq 0}{\varprojlim}\textrm{ }\tau^{\geq -i}X_n^\bullet)=H^j(\underset{n\geq 0}{\varprojlim}\textrm{ }X_n^\bullet)\qquad H^j(\underset{n\geq 0}{\varprojlim}\textrm{ }\tau^{\geq -i}Y_n^\bullet)=
H^j(\underset{n\geq 0}{\varprojlim}\textrm{ }Y_n^\bullet)
\end{equation} for all integers $j\geq -i+2$. Choosing $i$ to be arbitrarily large, we now see that
$H^j(X^\bullet)\overset{\sim}{\longrightarrow} H^j(Y^\bullet)$ for all $j\in \mathbb Z$, i.e., $g^\bullet:X^\bullet
\longrightarrow Y^\bullet$ is a quasi-isomorphism. 

\end{proof}

\medskip 
We now make one more observation. Let $\{I^{(j)}\}_{j\geq 0}$ be a family of pure injective modules and consider the product
$I:=\prod_{j\geq 0}I^{(j)}$. If $F$ is a finitely presented $R$-module and $A^\bullet\in C(R)$ is pure acyclic, we observe
that $Hom^\bullet_R(F,Hom_R^\bullet(A^\bullet,I))=Hom^\bullet_R(F\otimes_RA^\bullet,I)=\prod_{j\geq 0}
Hom_R^\bullet(F\otimes_RA^\bullet,I^{(j)})$ is acyclic. It follows that the product $I$ is also pure injective.
We now prove the main result of this section. 

\medskip
\begin{thm}\label{Thinj} Let $R$ be a commutative ring and let $M^\bullet\in C(R)$ be a cochain complex of $R$-modules. Then,
$M^\bullet$ has a pure injective resolution. 
\end{thm}

\begin{proof} We consider the special inverse system $\{I_n^\bullet\}_{n\geq 0}$ of $K$-pure injective complexes
and the pure quasi-isomorphisms $\{f_n:\tau^{\geq -n}M^\bullet\longrightarrow I_n^\bullet\}_{n\geq 0}$
given by  Proposition \ref{P3.9}. We choose any  pure projective $R$-module $P$. Then, we have
quasi-isomorphisms:
\begin{equation}
Hom_R^\bullet(P,f_n):Hom_R^\bullet(P,\tau^{\geq -n}M^\bullet)\longrightarrow 
Hom_R^\bullet(P,I_n^\bullet) 
\end{equation} We first consider the inverse system $\{Hom_R^\bullet(P,\tau^{\geq -n}M^\bullet)\}_{n\geq 0}$. 
Choose any $i\geq 0$. Then, it is clear that for $n\geq i+2$, the canonical morphisms 
\begin{equation}\label{3.30q}
\tau^{\geq -i}Hom_R^\bullet(P,\tau^{\geq -n-1}M^\bullet)\longrightarrow 
\tau^{\geq -i}Hom_R^\bullet(P,\tau^{\geq -n}M^\bullet)
\end{equation} are all identity. Further, from the proof of Proposition \ref{P3.9}, we know that each $I_{n+1}^\bullet$ 
can be expressed as $I_{n+1}^\bullet
=C_{x_n}[-1]^\bullet$, where $C_{x_n}^\bullet$
is   the mapping cone of a morphism $x_n^\bullet:I_n^\bullet\longrightarrow J_n^\bullet$. Since the functor
$Hom_R^\bullet(P,\_\_)$ commutes with mapping cones (see, for instance, \cite[(A.2.1.2)]{LCs}), we get 
\begin{equation}
Hom_R^\bullet(P,I_{n+1}^\bullet)=Cone(Hom_R^\bullet(P,x_n^\bullet))[-1]^\bullet
\end{equation} Accordingly, the morphisms $Hom_R^\bullet(P,I_{n+1}^\bullet)=Cone(Hom_R^\bullet(P,x_n^\bullet))[-1]^\bullet\longrightarrow Hom_R^\bullet(P,I_{n}^\bullet)$ are all surjective for $n\geq 0$. It follows easily
that for any $i\geq 0$, the induced morphism on the truncations:
\begin{equation}\label{3.32q}
\tau^{\geq -i}Hom_R^\bullet(P,I_{n+1}^\bullet)\longrightarrow \tau^{\geq -i}Hom_R^\bullet(P,I_{n}^\bullet)
\end{equation} is surjective for each $n\geq 0$. We now set $I^\bullet:=\underset{n\geq 0}{\varprojlim}\textrm{ } 
I_n^\bullet$ and consider the induced morphism $f^\bullet:M^\bullet
=\underset{n\geq 0}{\varprojlim}\textrm{ }\tau^{\geq -n}M^\bullet\longrightarrow I^\bullet$. From \eqref{3.30q} and \eqref{3.32q} and applying
Lemma \ref{Lem3.10}, it follows that we have a quasi-isomorphism on inverse limits:
\begin{equation}\label{3.33q}
\begin{CD}
Hom_R^\bullet(P,M^\bullet)=\underset{n\geq 0}{\varprojlim}\textrm{ }Hom_R^\bullet(P,\tau^{\geq -n}M^\bullet)\\
@VHom_R^\bullet(P,f^\bullet)VV \\
\underset{n\geq 0}{\varprojlim}\textrm{ }Hom_R^\bullet(P,I_n^\bullet)=Hom_R^\bullet(P,I^\bullet)  \\
\end{CD}
\end{equation} for any  pure projective $R$-module $P$.  From \eqref{3.33q}, we conclude that $f^\bullet:M^\bullet\longrightarrow I^\bullet$
is a pure quasi-isomorphism. Further, we know from Proposition \ref{P3.9} that $\{I_n^\bullet\}_{n\geq 0}$
is a special inverse system of 
$K$-pure injective complexes. It now follows from Proposition \ref{P3.8t}(b) that the limit
$I^\bullet$ of the special inverse system $\{I_n^\bullet\}_{n\geq 0}$  is still $K$-pure injective. It remains to show that 
$I^\bullet$ is a complex of pure injectives. For this, we notice that for any $j\in \mathbb Z$, 
$I^j$ is the limit of the following inverse system: 
\begin{equation}
\begin{CD}
...\longrightarrow I^j_{n+1}=I^j_{n}\oplus J_n^{j-1}@>p_{n+1}>> I_n^j=I^j_{n-1}\oplus J_{n-1}^{j-1}
@>p_n>> ...\longrightarrow I_0^j
\end{CD}
\end{equation} where each $p_n$ is the canonical projection onto the direct summand. Then, 
$I^j$ can be expressed as the direct product $I^j=I_0^j\oplus \prod_{n\geq 1}J_{n-1}^{j-1}$ of pure injectives. Hence,
$I^j$ is pure injective. 
\end{proof} 

\medskip

\medskip
\section{Pure projective resolutions of unbounded complexes}

\medskip

\medskip
In this section, we will construct pure projective resolutions for arbitrary complexes of $R$-modules. As in the previous
section, our methods are an adaptation of the classical method for constructing projective resolutions of 
bounded above complexes (see, for example, \cite{Murf1}) along with the techniques of Spaltenstein \cite{Sp} for treating unbounded complexes. Unfortunately,  the proofs in this section are not always the dual of the arguments in Section 3. However,  we will try to be as concise
as possible 
 by pointing out all those arguments that are dual to the case of pure injective resolutions. 

\medskip
\begin{defn}\label{Kproj} We will say that a cochain complex $M^\bullet\in C(R)$ is $K$-pure projective if the functor
$Hom_R^\bullet(M^\bullet,\_\_):C(R)\longrightarrow C(R)$ carries pure acyclic complexes to pure acyclic complexes. 
\end{defn}

\medskip
We make the following observation: given a pure acyclic complex $A^\bullet\in C(R)$  and any finitely presented
$R$-module $F$, we consider the complex $Hom_R^\bullet(F,A^\bullet)$. Now if $F'$ is any other
finitely presented $R$-module, the tensor product $F'\otimes_RF$ is still finitely presented. Hence, 
$Hom_R^\bullet(F'\otimes_RF,A^\bullet)$ is acyclic. Therefore, $
 Hom_R^\bullet(F',Hom_R^\bullet(F,A^\bullet))\cong Hom_R^\bullet(F'\otimes_RF,A^\bullet)$ is acyclic
for any finitely presented $R$-module $F'$ and we conclude that  $Hom_R^\bullet(F,A^\bullet)$ is actually
pure acyclic. 

\medskip
\begin{thm}\label{P4.2} For a complex $M^\bullet\in C(R)$, the following statements are equivalent:

\medskip
(a) $M^\bullet$ is $K$-pure projective. 

\smallskip
(b) For any pure acyclic complex $A^\bullet\in C(R)$, we have $Hom_{K(R)}(M^\bullet, A^\bullet)=0$. 

\smallskip
(c) For any complex $A^\bullet\in C(R)$, we have $Hom_{K(R)}(M^\bullet,A^\bullet)\cong 
Hom_{D_{pur}(R)}(M^\bullet,A^\bullet)$. 

\end{thm}

\begin{proof} (a) $\Rightarrow$ (b) : This is dual to the corresponding argument in the proof of Proposition \ref{P3.15}. 

\smallskip
(b) $\Rightarrow$ (a): Let $A^\bullet\in C(R)$ be pure acyclic. We have to show that 
$Hom_R^\bullet(M^\bullet,A^\bullet)$ is pure acyclic. From the observation above, we know that for any finitely presented $R$-module $F$, $Hom_R^\bullet(F,A^\bullet)$ (and hence $Hom_R^\bullet(F,A^\bullet)[k]$ for any $k\in 
\mathbb Z$) is pure acyclic. Then, we get:
\begin{equation}
\begin{array}{ll}
H^k(Hom_R^\bullet(F,Hom_R^\bullet(M^\bullet,A^\bullet)))&=H^k(Hom_R^\bullet(M^\bullet,
Hom_R^\bullet(F,A^\bullet)))\\ & =Hom_{K(R)}(M^\bullet,Hom_R^\bullet(F,A^\bullet)[k])=0\\
\end{array}
\end{equation} for any $k\in \mathbb Z$. Hence, $Hom_R^\bullet(M^\bullet,A^\bullet)$
is pure acyclic. 

\smallskip
(c) $\Rightarrow$ (b): This is clear because any pure acyclic $A^\bullet$ is $0$ in $D_{pur}(R)$. 

\smallskip
(a) $\Rightarrow$ (c): The proof is dual to the corresponding argument in the proof of Proposition \ref{P3.16}
if we consider the morphisms in the derived category $D_{pur}(R)$ as ``left roofs'' in place
of the ``right roofs'' appearing in \eqref{rroof1}.

\end{proof}

\medskip
\begin{thm}\label{P4.3} Let $P^\bullet\in C(R)$ be a bounded above complex of pure projective modules. 
Then, $P^\bullet$ is $K$-pure projective. 
\end{thm}

\begin{proof} Following Proposition \ref{P4.2}, it suffices to show that $Hom_{K(R)}(P^\bullet,A^\bullet)=0$
for any pure acyclic $A^\bullet\in C(R)$. The rest of the argument is now dual to that in the proof of 
Proposition \ref{P3.2}.

\end{proof}

\medskip
It is known (see \cite[Proposition 1]{War}) that for any $R$-module $M$, there exists a pure epimorphism
$P\twoheadrightarrow M$ with $P$ a pure projective module.  
We are now ready to show that any bounded above complex of $R$-modules admits a pure projective resolution. 
Now, a key step in the proof of Proposition \ref{P3.3}, which gives the corresponding result for pure injective resolutions, 
is the fact that the functor $\_\_\otimes_RN$ preserves cokernels for any $N\in R-Mod$. Since this no longer
holds for kernels, we must modify our approach somewhat to obtain pure projective resolutions, i.e.,  we cannot simply dualize the  proof of Proposition \ref{P3.3} here. 

\medskip
\begin{thm}\label{P4.4} (a) Let $M^\bullet\in C(R)$ be a cochain complex that is bounded above. Then, 
there exists a pure quasi-isomorphism $v^\bullet:P^\bullet\longrightarrow M^\bullet$ such that
$P^\bullet$ is a bounded above complex of pure projective modules.

\medskip
(b)  Let $f^\bullet:M_2^\bullet\longrightarrow M_1^\bullet $ be a morphism of bounded above complexes
in $C(R)$. Let $v^\bullet_2: P_2^\bullet\longrightarrow M_2^\bullet$ be a pure quasi-isomorphism to
$M_2^\bullet$ from  a bounded above complex $P_2^\bullet$ of pure projectives. Then, there exists a bounded above complex $P_1^\bullet$ of pure projectives,  a pure quasi-isomorphism
$v_1^\bullet: P_1^\bullet\longrightarrow M_1^\bullet$  and a morphism $g:P_2^\bullet\longrightarrow P_1^\bullet$  
fitting into a commutative diagram: 
\begin{equation}\label{tr3.12x}
\begin{CD}
P_2^\bullet @>v_2^\bullet>> M_2^\bullet \\
@Vg^\bullet VV @Vf^\bullet VV\\
P_1^\bullet @>v_1^\bullet>>  M_1^\bullet\\
\end{CD}
\end{equation}
\end{thm}

\begin{proof} (a) For the sake of definiteness, we suppose  that the complex $(M^\bullet,d^\bullet)$
 satisfies $M^0\ne 0$ and $M^i =0$ for each $i>0$. We set $P^i=0$ for each $i>0$ and choose
a pure epimorphism $v^0:P^0\longrightarrow M^0$ with $P^0$ pure projective. Then, for every 
$i\geq 0$, we have already obtained pure projectives $P^i$, morphisms $v^i:P^i\longrightarrow M^i$ along with differentials
$e^i:P^i\longrightarrow P^{i+1}$ such that we have induced isomorphisms:
\begin{equation}
Ker(Hom_R(Q,e^{i+1}))/Im(Hom_R(Q,e^i))\overset{\sim}{\longrightarrow} H^{i+1}(Hom_R^\bullet(Q,M^\bullet) ) 
\qquad\forall\textrm{ }Q\in \mathcal{PP}
\end{equation} 
and epimorphisms: 
\begin{equation}
Ker(Hom_R(Q,e^{i}))\twoheadrightarrow Ker(Hom_R(Q,d^{i}))\qquad\forall\textrm{ }Q\in \mathcal{PP}
\end{equation}We suppose that we have already done this  for all integers $i$ greater than some given integer
$k\geq -1$. We will show that there exists a pure projective $P^k$, a morphism $v^k:P^k\longrightarrow M^k$ and
a differential $e^k:P^k\longrightarrow P^{k+1}$
such that: 
\begin{equation}\label{eq4.3}
\begin{array}{c}
Ker(Hom_R(Q,e^{k+1}))/Im(Hom_R(Q,e^k))\overset{\sim}{\longrightarrow} H^{k+1}(Hom_R^\bullet(Q,M^\bullet) ) \qquad\forall\textrm{ }Q\in \mathcal{PP}\\  
Ker(Hom_R(Q,e^{k}))\twoheadrightarrow Ker(Hom_R(Q,d^{k}))\qquad\forall\textrm{ }Q\in \mathcal{PP}\\
\end{array}
\end{equation} We now consider the object $L^k$ defined by
the following pullback square:
\begin{equation}
\begin{CD}
L^k @>>> Ker(e^{k+1}) \\
@VVV @VVV \\
M^k @>d^k>> M^{k+1} \\
\end{CD}
\end{equation} and choose a pure epimorphism $P^k\longrightarrow L^k$ with $P^k$ pure projective. Then, for any $Q\in 
\mathcal{PP}$, the induced morphism $Hom_R(Q,P^k)\longrightarrow Hom_R(Q,L^k)$ is still an epimorphism and we obtain
a pullback square
\begin{equation}\label{pull4}
\begin{CD}
Hom_R(Q,L^k) @>>> Hom_R(Q,Ker(e^{k+1}))=Ker(Hom_R(Q,e^{k+1})) \\
@VVV @VVV \\
Hom_R(Q,M^k) @>Hom_R(Q,d^k)>> Hom_R(Q,M^{k+1}) \\
\end{CD}
\end{equation} We  set $e^k:P^k\longrightarrow P^{k+1}$ to be the composition
$P^k\longrightarrow L^k\longrightarrow Ker(e^{k+1})\longrightarrow P^{k+1}$ and 
$v^k:P^k\longrightarrow M^k$ to be the composition $P^k\longrightarrow L^k\longrightarrow M^k$. Now applying 
to the pullback square \eqref{pull4} arguments that are dual to those applied to the pushout square
\eqref{3.6y} in the proof of Proposition \ref{P3.3}, we obtain the required results
in \eqref{eq4.3}.  The induced morphisms $Hom_R^\bullet(Q,v^\bullet):Hom_R^\bullet(Q,P^\bullet)\longrightarrow 
Hom_R^\bullet(Q,M^\bullet)$ being quasi-isomorphisms for each pure projective
$Q\in\mathcal{PP}$, it follows that $v^\bullet:P^\bullet\longrightarrow M^\bullet$ is a pure
quasi-isomorphism. 

\medskip
(b) As in part (a), we use the functors $Hom_R(Q,\_\_)$ with $Q\in \mathcal{PP}$ in place of the functors
$\_\_\otimes_RN$ with $N\in R-Mod$ appearing in the proof of Proposition \ref{P3.5}. Since
$Hom_R(Q,\_\_)$ also preserves mapping cones, we can now apply arguments dual 
to those in the proof of Proposition \ref{P3.5} to prove this result. 
\end{proof} 

\medskip
In order to proceed to pure projective resolutions of unbounded complexes, we will need to 
consider ``special direct systems''. 

\medskip
\begin{defn}\label{spproj} (see \cite[Definition 2.6]{Sp} )  Let $\mathcal T$ be a class of complexes in $C(R)$ that
is closed under isomorphisms.

\medskip

 (a) A $\mathcal T$-special direct system
of complexes  is a direct system  $\{P_n^\bullet\}_{n\in \mathbb Z}$ of complexes in $C(R)$ satisfying the 
following conditions for each $n\in \mathbb Z$:

\medskip
(1)  The cochain map $P_{n-1}^\bullet\longrightarrow P^\bullet_{n}$ is injective. 

\smallskip
(2) The cokernel $C_n^\bullet:=Coker(P_{n-1}^\bullet\longrightarrow P_{n}^\bullet)$
lies in the class $\mathcal T$. 

\smallskip
(3) The short exact sequence of complexes:
\begin{equation}
0\longrightarrow P^\bullet_{n-1} \overset{i_n^\bullet}{\longrightarrow} P_n^\bullet\overset{p^\bullet_n}{\longrightarrow} C_{n}^\bullet \longrightarrow 0
\end{equation} is ``semi-split'', i.e., it is split in each degree.

\medskip
(b) The class $\mathcal T\subseteq C(R)$ is said to be closed under special direct limits if the direct limit 
of every
$\mathcal T$-special direct system in $C(R)$  is contained in $\mathcal T$. 
\end{defn}

\medskip
By slight abuse of notation, we will  refer to  $\{P_n^\bullet\}_{n\geq 0}$ as a special direct
system if setting $P_n^\bullet=0$ for all $n<0$ makes $\{P_n^\bullet\}_{n\in \mathbb Z}$ 
into a special direct system in the sense of Definition \ref{spproj} above.

\medskip
\medskip
\begin{thm}\label{P4.8t} (a) Let $\mathcal C\subseteq C(R)$ be a class of complexes. Let $\mathcal T(\mathcal C)$ denote the class of 
complexes  $M^\bullet\in C(R)$  such that $Hom_R^\bullet(M^\bullet,A^\bullet)$ is pure acyclic for
each $A^\bullet \in \mathcal C$. Then, $\mathcal T(\mathcal C)$ is closed under special direct limits. 

\medskip
(b) The class of all $K$-pure projective complexes is closed under special direct limits. 
\end{thm} 

\begin{proof} (a) We set:
\begin{equation*}
\mathcal T':=\{\mbox{$B^\bullet\in C(R)$ $\vert$ $B^\bullet=Hom_R^\bullet(F,A^\bullet)$ for some finitely presented
$R$-module $F$ \& some $A^\bullet\in \mathcal C$}\}
\end{equation*} Now since $Hom_R^\bullet(F,Hom_R^\bullet(M^\bullet,A^\bullet))
\cong Hom_R^\bullet (M^\bullet,Hom_R^\bullet(F,A^\bullet))$, it follows that  a complex $M^\bullet\in \mathcal T(\mathcal C)$
if and only if $Hom_R^\bullet(M^\bullet,B^\bullet)$ is acyclic for each $B^\bullet\in \mathcal T'$. From 
 \cite[Corollary 2.8]{Sp}, we now see  that $\mathcal T(\mathcal C)$ is closed under special direct limits. Finally, the result of 
(b) follows  from (a) by letting $\mathcal C$  be the class of all pure acyclic complexes
in $C(R)$. 

\end{proof}

\medskip
For a complex $(M^\bullet,d^\bullet)\in C(R)$ and any given integer $n$, we now recall that its truncation
$\tau_{\leq n}M^\bullet$ is given by:

\begin{equation}\label{abtrunc}
(\tau_{\leq n}M)^i=\left\{\begin{array}{ll}
0 & \mbox{if $i>n$} \\
Ker(d^n) & \mbox{if $i=n$} \\
M^i & \mbox{if $i<n$} \\
\end{array}\right.
\end{equation}
It is clear that the complex $M^\bullet$ may be expressed as the direct limit 
$M^\bullet=\underset{n\geq 0}{\varinjlim}\textrm{ }\tau_{\leq n}M^\bullet$. 

\medskip

\begin{thm}\label{P4.9} For any complex $M^\bullet\in C(R)$ there exists a special direct system $\{P_n^\bullet\}_{n\geq 0}$
of $K$-pure projective complexes and a morphism $\{f_n:P_n^\bullet\longrightarrow \tau_{\leq n}M^\bullet\}_{n\geq 0}$ 
of direct systems 
satisfying the following conditions:

\medskip
(a) Each $P_n^\bullet$ is a bounded above complex of pure projectives. 

\medskip
(b) Each $f_n$ is a pure quasi-isomorphism. 
\end{thm}

\begin{proof} The proof of this is dual to that of Proposition \ref{P3.9}. 

\end{proof} 

\medskip
\begin{lem}\label{Lem4.8} Let $\{P^{(j)}\}_{j\geq 0}$ be a family of pure projective modules. Then, the 
direct sum $\underset{j\geq 0}{\bigoplus}P^{(j)}$ is pure projective.
\end{lem}

\begin{proof}  We set $P:=\underset{j\geq 0}{\bigoplus}P^{(j)}$ and consider a pure acyclic complex
$A^\bullet$. We have to check that $Hom_R^\bullet(P,A^\bullet)$ is pure acyclic. For this, we choose
a finitely presented $R$-module $F$ and see that:
\begin{equation*}
Hom_R^\bullet(F, Hom_R^\bullet(P,A^\bullet))= Hom_R^\bullet(F,\underset{j\geq 0}{\prod}Hom_R^\bullet(
P^{(j)},A^\bullet))=\underset{j\geq 0}{\prod}Hom_R^\bullet(F, Hom_R^\bullet(P^{(j)},A^\bullet))
\end{equation*} Since each $P^{(j)}$ is pure projective, each of the complexes $Hom_R^\bullet(P^{(j)},A^\bullet)$ is pure acyclic 
and hence each $Hom_R^\bullet(F, Hom_R^\bullet(P^{(j)},A^\bullet))$ is acyclic. Since the product of acyclic
complexes in $R-Mod$ must be acyclic, we conclude that $Hom_R^\bullet(F, Hom_R^\bullet(P,A^\bullet))$ is acyclic for
any finitely presented $R$-module $F$. This proves the result. 

\end{proof} 

\medskip
\medskip
\begin{thm}\label{Thproj} Let $R$ be a commutative ring and let $M^\bullet\in C(R)$ be a cochain complex of $R$-modules. Then,
$M^\bullet$ has a pure projective resolution. 
\end{thm}

\begin{proof} We consider the special direct system $\{P_n^\bullet\}_{n\geq 0}$ of $K$-pure projective complexes along with the pure quasi-isomorphisms
$\{f_n:P_n^\bullet\longrightarrow \tau_{\leq n}M^\bullet\}_{n\geq 0}$  from the proof of Proposition \ref{P4.9}. 
We set $P^\bullet:=\underset{n\geq 0}{\varinjlim}\textrm{ }P_n^\bullet$. It is clear that pure quasi-isomorphisms commute
with direct limits. This gives us an induced pure quasi-isomorphism:
\begin{equation}
f^\bullet: P^\bullet=\underset{n\geq 0}{\varinjlim}\textrm{ }P_n^\bullet\longrightarrow \underset{n\geq 0}{\varinjlim}\textrm{ }
\tau_{\leq n}M^\bullet=M^\bullet
\end{equation} Using Proposition \ref{P4.8t}(b),
we  see that the direct limit $P^\bullet$ is  $K$-pure projective. It remains to show that $P^\bullet$ is a complex
of pure projective modules. From the construction of each $P^\bullet_n$ in Proposition \ref{P4.9} which is dual
to the construction in Proposition \ref{P3.9}, it follows that each term in  the direct limit $P^\bullet$ is actually a direct sum of a family
of pure projective modules. It now follows from Lemma \ref{Lem4.8} that each term in $P^\bullet$ is pure projective.
\end{proof} 

\medskip

\medskip
\section{Closed monoidal structure on the pure derived category}

\medskip

\medskip
In this section, we will use the pure projective and pure injective resolutions developed so far  to give a closed
monoidal structure on the pure derived category $D_{pur}(R)$. Now, if $u^\bullet: M^\bullet\longrightarrow 
N^\bullet$ is a pure quasi-isomorphism and $A$ is any $R$-module, it is immediate from the definitions
that the induced morphism $A\otimes_Ru^\bullet:A\otimes_RM^\bullet\longrightarrow A\otimes_RN^\bullet$
is a pure quasi-isomorphism. In order to produce a tensor structure on $D_{pur}(R)$, we will need to extend
this fact to tensor products of (possibly unbounded) complexes of $R$-modules. 

\medskip
Given cochain complexes $(M^\bullet,d_M^\bullet)$, $(A^\bullet,d_A^\bullet)\in C(R)$, we recall that their tensor product $(M^\bullet\otimes_RA^\bullet)^\bullet$
is the total complex associated to the double complex $(M^\bullet\otimes_RA^\bullet)^{ij}:=M^i\otimes_RA^j$ with differentials
$d_1^{ij}:=d_M^i\otimes_RA^j:(M^\bullet\otimes_RA^\bullet)^{ij}
\longrightarrow (M^\bullet\otimes_RA^\bullet)^{i+1,j}$ and $d_2^{ij}:=M^i\otimes_Rd_A^j:(M^\bullet\otimes_RA^\bullet)^{ij}
\longrightarrow (M^\bullet\otimes_RA^\bullet)^{i,j+1}$.  

\medskip
\begin{lem}\label{Lm5.1} Let $u^\bullet: M^\bullet\longrightarrow 
N^\bullet$ be a pure quasi-isomorphism. Then, for any given integer $n\in \mathbb Z$,
the induced morphism on the truncations $\tau_{\leq n}u^\bullet:\tau_{\leq n}M^\bullet\longrightarrow \tau_{\leq n}
N^\bullet$ is a pure quasi-isomorphism. 
\end{lem}

\begin{proof} We will show that  $Hom_R^\bullet(P,\tau_{\leq n}u^\bullet):Hom_R^\bullet(P,\tau_{\leq n}M^\bullet)
\longrightarrow Hom_R^\bullet(P,\tau_{\leq n}N^\bullet)$ is a quasi-isomorphism for any pure projective
$P\in \mathcal{PP}$. Since $u^\bullet$ is a pure quasi-isomorphism, we already know that 
$Hom_R^\bullet(P,u^\bullet):Hom_R^\bullet(P,M^\bullet)\longrightarrow Hom_R^\bullet
(P,N^\bullet)$ is a quasi-isomorphism. This induces a quasi-isomorphism 
on the truncations:
\begin{equation}\label{5.1dsm}
\tau_{\leq n}Hom_R^\bullet(P,u^\bullet):\tau_{\leq n}Hom_R^\bullet(P,M^\bullet)\longrightarrow 
\tau_{\leq n}Hom_R^\bullet
(P,N^\bullet)
\end{equation} Further, since the functor $Hom_R(P,\_\_)$ preserves kernels, it is clear from
the definition in \eqref{abtrunc} that the truncations satisfy:
\begin{equation}\label{5.2dsm}
Hom_R^\bullet(P,\tau_{\leq n}M^\bullet)=\tau_{\leq n}Hom_R^\bullet(P,M^\bullet)
\qquad Hom_R^\bullet(P,\tau_{\leq n}N^\bullet)=\tau_{\leq n}Hom_R^\bullet
(P,N^\bullet) 
\end{equation} for any integer $n\in \mathbb Z$. 
Combining \eqref{5.1dsm} and \eqref{5.2dsm}, the result follows. 
\end{proof}

\medskip
\begin{thm}\label{P5.2} Let $u^\bullet:M^\bullet\longrightarrow N^\bullet$ be a pure quasi-isomorphism of 
 complexes. Then, for any cochain complex $A^\bullet\in C(R)$, the induced morphism
$A^\bullet\otimes_Ru^\bullet: (A^\bullet\otimes_R M^\bullet)^\bullet\longrightarrow (A^\bullet\otimes_RN
^\bullet)^\bullet$ is a pure 
quasi-isomorphism. 
\end{thm}

\begin{proof} We choose any $R$-module $B$ and set $C^\bullet:=B\otimes_RA^\bullet$. Further, for any
integer $m\in \mathbb Z$, we set $C_m^\bullet:=\tau_{\leq m}C^\bullet$. Now for any $m,n\in \mathbb Z$, we claim that
the induced morphism:
\begin{equation}\label{eq5.3}
(C_m^\bullet\otimes_R\tau_{\leq n}M^\bullet)^\bullet\longrightarrow (C_m^\bullet\otimes_R\tau_{\leq n}N^\bullet)^\bullet
\end{equation} is a quasi-isomorphism. From Lemma \ref{Lm5.1}, we know that $
\tau_{\leq n}u^\bullet:\tau_{\leq n}M^\bullet
\longrightarrow \tau_{\leq n}N^\bullet$ is a pure quasi-isomorphism. It follows that for any fixed $i\in \mathbb Z$, 
the morphism
\begin{equation}\label{eq5.4}
C_m^i\otimes_R\tau_{\leq n}u^\bullet : C_m^i\otimes_R(\tau_{\leq n}M^\bullet)\longrightarrow C_m^i\otimes_R(\tau_{\leq n}N^\bullet)
\end{equation} is  a quasi-isomorphism. Since $C_m^\bullet$, $\tau_{\leq n}M^\bullet$ and
$\tau_{\leq n}N^\bullet$ are all bounded above complexes, it now follows from a standard spectral sequence
argument (see, for instance, \cite[Tag 0132]{Stacks}) that the quasi-isomorphisms in \eqref{eq5.4} induce a quasi-isomorphism of the total
complexes in \eqref{eq5.3}. Taking  direct limits of quasi-isomorphisms over all $m$, $n\in \mathbb Z$, 
we obtain a quasi-isomorphism:
\begin{equation}
B\otimes_R(A^\bullet\otimes_RM^\bullet)^\bullet=(C^\bullet\otimes_RM^\bullet)^\bullet\longrightarrow (C^\bullet\otimes_RN^\bullet)^\bullet=B\otimes_R(A^\bullet\otimes_RN^\bullet)^\bullet
\end{equation} for any $R$-module $B$. This proves the result. 

\end{proof}

\medskip
From Proposition \ref{P5.2}, it is clear that the tensor product $\otimes : C(R)\times C(R)
\longrightarrow C(R)$ descends to a tensor product:
\begin{equation}\label{5tensor}
\otimes: D_{pur}(R)\times D_{pur}(R)\longrightarrow D_{pur}(R)
\end{equation} on the pure derived category of $R$-modules.

\medskip
\begin{thm}\label{P5.3} Let $u^\bullet:M^\bullet\longrightarrow N^\bullet$ be a pure quasi-isomorphism 
of complexes. Then, the following hold:

\smallskip
(a) Let $P^\bullet$ be a $K$-pure projective complex. Then, 
$Hom_R^\bullet(P^\bullet,u^\bullet):Hom_R^\bullet(P^\bullet,M^\bullet)\longrightarrow 
Hom_R^\bullet(P^\bullet,N^\bullet)$ is a pure quasi-isomorphism. 

\smallskip
(b) Let $I^\bullet$ be a $K$-pure injective complex. Then,
$Hom_R^\bullet(u^\bullet,I^\bullet):Hom_R^\bullet(N^\bullet,I^\bullet)
\longrightarrow Hom_R^\bullet(M^\bullet,I^\bullet)$ is a pure quasi-isomorphism. 
\end{thm}

\begin{proof} (a) Let $C_u^\bullet$ denote the pure acyclic complex that is the mapping cone of the 
pure quasi-isomorphism 
$u^\bullet:M^\bullet\longrightarrow N^\bullet$. For any complex $B^\bullet\in C(R)$, we know that
the  functor $Hom_R^\bullet(B^\bullet,\_\_)$ on $C(R)$ commutes
with mapping cones (see, for instance, \cite[(A.2.1.2)]{LCs}). In particular, we see that:
\begin{equation}\label{eq5.7}
Cone(Hom_R^\bullet(P^\bullet,u^\bullet))=Hom_R^\bullet(P^\bullet,C_u^\bullet) 
\end{equation} Since $P^\bullet$ is a $K$-pure projective complex, it follows from 
the definition that $Hom_R^\bullet(P^\bullet,C_u^\bullet) $ is pure acyclic. Combining
this with \eqref{eq5.7}, we conclude that $Cone(Hom_R^\bullet(P^\bullet,u^\bullet))$ is 
pure acyclic and hence $Hom_R^\bullet(P^\bullet,u^\bullet)$ is a pure quasi-isomorphism. 

\medskip
(b) For any complex $B^\bullet\in C(R)$, 
the  contravariant functor $Hom_R^\bullet(\_\_,B^\bullet)$ on $C(R)$ preserves mapping cones
up to a shift (see, for instance, \cite[(1.5.3)]{Lip}). Since the shift of a pure acylic complex is
still pure acyclic, the result of part (b) now follows by applying an argument dual 
to that in part (a). 

\end{proof}

\medskip
Let $M^\bullet$, $N^\bullet\in C(R)$ be two arbitrary complexes of $R$-modules. From the results
of Section 4, we may choose a pure quasi-isomorphism $P_M^\bullet\longrightarrow M^\bullet$ giving a 
pure projective resolution of $M^\bullet$. Similarly, using the results of Section 3, we may choose
a pure quasi-isomorphism $N^\bullet\longrightarrow I_N^\bullet$ giving a pure injective resolution
of $N^\bullet$. We can now define a ``pure derived Hom functor'':
\begin{equation}\label{5hom}
\begin{array}{c}
PHom^\bullet_R(\_\_,\_\_): D_{pur}(R)^{op}\times D_{pur}(R)\longrightarrow D_{pur}(R)\\
PHom^\bullet_R(M^\bullet,N^\bullet):=Hom_R^\bullet(P_M^\bullet,I_N^\bullet) \\
\end{array}
\end{equation} The functor in \eqref{5hom} is well defined due to the pure quasi-isomorphisms
\begin{equation}
Hom_R^\bullet(P_M^\bullet,N^\bullet)\longrightarrow Hom_R^\bullet(P_M^\bullet,I_N^\bullet)\longleftarrow Hom_R^\bullet(M^\bullet,I_N^\bullet)
\end{equation} that both follow from Proposition \ref{P5.3}. 

\medskip
We conclude by showing that we have obtained a closed symmetric monoidal structure 
on the pure derived category $D_{pur}(R)$. 

\medskip
\begin{thm}\label{P5.6} For any $A^\bullet$, $B^\bullet$, $C^\bullet\in C(R)$, we have natural
isomorphisms:
\begin{equation}\label{qe5.10}
\begin{array}{c}
PHom_R^\bullet((A^\bullet\otimes_RB^\bullet)^\bullet,C^\bullet)\cong PHom_R^\bullet (A^\bullet,P
Hom^\bullet_R(B^\bullet,C^\bullet))\\
Hom_{D_{pur}(R)}((A^\bullet\otimes_RB^\bullet)^\bullet,C^\bullet)\cong Hom_{D_{pur}}(A^\bullet,P
Hom^\bullet_R(B^\bullet,C^\bullet))\\
\end{array}
\end{equation}

\end{thm}

\begin{proof} We choose pure projective resolutions $P_A^\bullet$ and $P_B^\bullet$ of 
$A^\bullet$ and $B^\bullet$ respectively. Let $I_C^\bullet$ be a pure injective resolution
of $C^\bullet$. From Proposition \ref{P5.2}, it follows that $(P_A^\bullet\otimes_RP_B^\bullet)^\bullet$
is pure quasi-isomorphic to $(A^\bullet\otimes_RB^\bullet)^\bullet$. From the definitions, it is also clear
that $(P_A^\bullet\otimes_RP_B^\bullet)^\bullet$ is a $K$-pure projective complex each of the terms of which
is pure projective. Hence, $(P_A^\bullet\otimes_RP_B^\bullet)^\bullet$ is a pure projective resolution 
of $(A^\bullet\otimes_RB^\bullet)^\bullet$. It now follows that:
\begin{equation}
\begin{array}{ll}
PHom_R^\bullet((A^\bullet\otimes_RB^\bullet)^\bullet,C^\bullet)&= Hom_R^\bullet((P_A^\bullet\otimes_RP_B^\bullet)^\bullet,I_C^\bullet)\\
& \cong Hom_R^\bullet(P_A^\bullet, Hom_R^\bullet(P_B^\bullet,I_C^\bullet))\\
& = Hom_R^\bullet(P_A^\bullet,PHom_R^\bullet(B^\bullet,C^\bullet))\\
&=PHom_R^\bullet(A^\bullet,PHom_R^\bullet(B^\bullet,C^\bullet))\\
\end{array} 
\end{equation} To prove the second isomorphism in \eqref{qe5.10}, we proceed as follows: it is  already
clear  that:
\begin{equation}\label{5.12qe}
Hom_{D_{pur}(R)}((A^\bullet\otimes_RB^\bullet)^\bullet,C^\bullet)\cong Hom_{D_{pur}(R)}((P_A^\bullet\otimes_RP_B^\bullet)^\bullet,I_C^\bullet)
\end{equation} Since $I_C^\bullet$ is $K$-pure injective, it follows from Proposition \ref{P3.16} that:
\begin{equation}\label{5.13qe}
Hom_{D_{pur}(R)}((P_A^\bullet\otimes_RP_B^\bullet)^\bullet,I_C^\bullet)\cong Hom_{K(R)}((P_A^\bullet\otimes_RP_B^\bullet)^\bullet,I_C^\bullet)
\end{equation} The closed monoidal structure on $K(R)$ now gives:
\begin{equation}\label{5.14qe}
Hom_{K(R)}((P_A^\bullet\otimes_RP_B^\bullet)^\bullet,I_C^\bullet)
\cong Hom_{K(R)}(P_A^\bullet,Hom_R^\bullet(P_B^\bullet,I_C^\bullet))
\end{equation} We can put $P
Hom^\bullet_R(B^\bullet,C^\bullet)=Hom_R^\bullet(P_B^\bullet,I_C^\bullet)$. Further, since
$P_A^\bullet$ is $K$-pure projective, it follows from Proposition \ref{P4.2} that:
\begin{equation}\label{5.15qe}
 Hom_{K(R)}(P_A^\bullet,Hom_R^\bullet(P_B^\bullet,I_C^\bullet))\cong 
Hom_{D_{pur}(R)}(P_A^\bullet,P
Hom^\bullet_R(B^\bullet,C^\bullet))
\end{equation} Since $P_A^\bullet\cong A^\bullet$ in $D_{pur}(R)$, the sequence of isomorphisms
\eqref{5.12qe}-\eqref{5.15qe} proves the second isomorphism in \eqref{qe5.10}. 

\end{proof}

\end{document}